\documentclass[reqno,oneside,12pt]{amsart}
\usepackage{amssymb,enumerate,bm}
\usepackage[T1]{fontenc}
\usepackage[width=17cm,centering]{geometry}
\usepackage[colorlinks=true, allcolors=blue]{hyperref}
\usepackage{tcolorbox, stmaryrd, mathrsfs}
\usepackage{color}
\usepackage{tikz}
\usepackage{parskip}
\usepackage{mathtools}
\usepackage[shortlabels]{enumitem}
\allowdisplaybreaks

\newtheorem{theorem}{Theorem}[section]

\newtheorem{cor}[theorem]{Corollary}
\newtheorem{lem}[theorem]{Lemma}
\newtheorem{prop}[theorem]{Proposition}

\theoremstyle{definition}

\theoremstyle{remark}
\newtheorem{remark}[theorem]{Remark}

\newtheorem{definition}[theorem]{Definition}

\newcommand{\R}{\mathbb{R}}

\newcommand{\dist}{\operatorname{dist}}
\newcommand{\side}{\operatorname{side}}

\newcommand{\conv}{\operatorname{conv}}
\newcommand{\area}{\operatorname{area}}
\newcommand{\set}[1]{\left\{#1\right\}}

\begin{document}
\title[Minimizers of MDP via an ATST Algorithm]{Minimizers of the Maximum Distance Problem via an Analyst's Traveling Salesperson Algorithm}
\author[1]{Enrique Alvarado}
\thanks{Department of Mathematics, Iowa State University, Ames IA, USA \texttt{enrique3@iastate.edu}}
\author[2]{Silvia Ghinassi}
\thanks{Rome, Italy \texttt{silvia.ghinassi@gmail.com}}
\author[3]{Lisa Naples}
\thanks{Mathematics Department, Fairfield University, Fairfield CT, USA \texttt{lisa.naples@fairfield.edu}}

\begin{abstract}
    We provide an upper and lower bound for the length of Maximum Distance Problem minimizers in terms of a finite scale geometric square sum. 
\end{abstract}
\maketitle

\section{Introduction}

Let $E \subset \R^2$ be a set. Although there are many sets in Euclidean space that cannot be contained in \emph{any} curve $\Gamma$ of finite length, for any $r > 0$, there are \emph{many} curves of finite length whose closed \textit{$r$-neighborhood} $\mathcal{N}(\Gamma, r) := \cup_{x \in \Gamma} B(x, r)$ contains every point in $E$. 

The \emph{Maximum Distance Problem} (MDP) is the problem of finding the shortest such curve:
\begin{equation}\label{eq: MDP}
\notag
\begin{cases}
    \text{minimize } \mathcal{H}^1(\Gamma) \\
    \text{among curves of finite length } \Gamma\subset \mathbb{R}^2 \text{ such that } \mathcal{N}(\Gamma, r) \supset E.
    \end{cases}
\end{equation} 
Here and throughout, we define a curve of finite length to be a connected set of finite 1-dimensional Hausdorff measure $\mathcal{H}^1$. 
Not knowing the existence of a minimizer a priori, we define
\[
\Lambda(E, r) := \inf\{\mathcal{H}^1(\Gamma) : \Gamma\subset \mathbb{R}^2 \text{ is a curve of finite length and } \mathcal{N}(\Gamma, r) \supset E\}.
\]
However, when $E$ is compact, standard compactness arguments prove the existence of minimizers (see \cite[Theorem 2.15]{AKV}).
The problem was first introduced by Buttazzo, Oudet, and Stepanov~\cite{BOS} in the context of the study optimal urban transportation networks in cities.

Alvarado, Catalano, Merch\'an, and Naples \cite{ACMN} show that if $E$ is contained in a compact connected set of finite $\mathcal{H}^1$-measure, that is, a curve of finite $\mathcal{H}^1$, then the $r$-maximum distance minimizers converge in both Hausdorff distance and $\mathcal{H}^1$-measure to a curve of smallest measure that contains our original set. 
This is reminiscent of a classical result in harmonic analysis, the so-called Analyst's Traveling Salesperson Theorem. In the early 1990s, Jones~\cite{J} and Okikiolu~\cite{O} obtained an important breakthrough with the celebrated \textit{Analyst's Traveling Salesperson Theorem} (ATST). 
They provided necessary and sufficient geometric conditions for a set in $\R^n$, which may be thought of as an infinite list of cities, to be contained in a curve (or path between cities) of finite length. 
Since then, the tools developed by Jones have garnered immense interest among geometric measure theorists, and many contributions to the study of qualitative and quantitative rectifiability rely on the ideas behind ATST.

Before introducing the Analyst's Traveling Salesperson Theorem, we need a few definitions. For a subset $E$ of $\R^2$, we define $|E|$ to be the diameter of $E$, that is, $|E| = \sup_{x,y \in E} \dist(x,y)$. 

\begin{definition}[Dyadic cubes]
We let $\mathcal{D}$ denote the collection of dyadic cubes; that is,
$$\mathcal{D} = \set{[k 2^{-n}, (k+1) 2^{-n}) \times [j 2^{-n}, (j+1) 2^{-n}) : k, j, n \in \mathbb{Z}}.$$
The collection of all dyadic cubes $Q$ for which $\side(Q) = 2^{-n}$ is denoted by
$$\mathcal{D}_n = \set{[k 2^{-n}, (k+1) 2^{-n}) \times [j 2^{-n}, (j+1) 2^{-n}) : k, j \in \mathbb{Z}}.$$    
\end{definition}

\begin{definition}[Hausdorff measure]
For any $E \subset\R^2$, $\delta > 0$ and $s \ge 0$, define 
$$\mathcal{H}^s_\delta(E) = \inf \set{ \sum _{i = 1}^\infty |U_i|^s : E \subset\bigcup_{i=1}^\infty U_i, |U_i| < \delta}.$$
The $s$-dimensional Hausdorff measure of a measurable set $E \subset\R^2$ is defined as
$$\mathcal{H}^s(E) = \lim_{\delta \downarrow 0} \mathcal{H}^s_\delta(E).$$ 
\end{definition}
\begin{definition}[$\beta$-numbers]
    Given $E \subset \R^2$ and $Q \in \mathcal{D}$, define
    \begin{align*}
        \beta_E(Q): = \begin{cases}
            \displaystyle \inf_{\ell \text{ line}} \sup_{x \in E \cap Q} |Q|^{-1} \mbox{dist}(x, \ell) & \text{ if } E \cap Q \neq \varnothing, \\
            0 & \text{ if } E \cap Q = \varnothing.
        \end{cases}
    \end{align*} 
\end{definition}

\begin{theorem}[Analyst's Traveling Salesperson Theorem {\cite[Theorem 1]{J}}] \label{ATST}
\label{atst}
    A bounded set $E \subset \R^2$ is contained in a curve of finite length if and only if
    \begin{align*}
         \sum_{\substack{Q \in \mathcal{D}}} \beta^2_E(3Q) |Q| < \infty.
    \end{align*}
    More precisely, if $\sum_{\substack{Q \in \mathcal{D}}} \beta^2_E(3Q) |Q|< \infty$, then there exist a curve of finite length $I$ and a universal constant $c_1>0$ so that $I \supset E$ and
    $\mathcal{H}^1(I) \leq c_1 \left(|E| + \sum_{\substack{Q \in \mathcal{D}}} \beta^2_E(3Q) |Q|\right)$. 
    If $I$ is any curve of finite length containing $E$, then there exists a universal constant $c_2 > 0$ such that $|E| + \sum_{\substack{Q \in \mathcal{D}}} \beta^2_E(3Q) |Q| \leq c_2 \mathcal{H}^1(I)$.
\end{theorem}

Motivated by the connection between the asymptotics of MDP minimizers  and ATST-type results for sufficiently regular sets $E$ established in \cite{ACMN}, we investigate in this paper the connection between finite scale square functions and the MDP. 

\begin{theorem}[Main theorem]
\label{thm: Main}
    Given a compact set $E\subset\R^2$,
\[
|E| - 2r + \sum_{Q \in \mathcal{D}} \max\{\beta_E(Q) - r/|Q|, 0\}^2|Q| \lesssim \Lambda(E, r) \lesssim |E|  + \sum_{Q \in \mathcal{D}} \max\{\beta_E(Q) - r/|Q|, 0\}^2|Q|.
\]
\end{theorem}

Here and throughout the paper we say that $A \lesssim B$ if there exists a constant $c>1$ such that $A \leq c B$.
Note that, in the above theorem, $\beta_E(Q)-r/|Q|\le 0$ for all $Q$ with $|Q|<r$. Therefore, the sums that appear here are nonzero for finitely many generations (at most $\left\lceil-\log(r)/\log(2)\right\rceil$ generations). 

\begin{remark}
    The lower and upper bound differ by an additive factor of $-2r$. We cannot improve the lower bound by removing such factor. In fact, consider the case where $E$ is made up of two points, at a distance $d > 2r$ apart. The segment contained in the line passing through the two points, of length $d-2r$, whose endpoints are at distance $r$ from each of the two points in $E$, is a minimizer for the MDP and has length exactly $|E|-2r$.
\end{remark}

Rather than Jones's original proof, our techniques rely heavily on a substantially different proof by Bishop and Peres, \cite{BP}. Bishop and Peres construct a nested family of convex sets, which act as neighborhoods enclosing $E$ at each fixed generation. For the reader's convenience, we include proofs for the results we use, in the notation of our construction and with details filled in.

\subsection{Acknowledgments} This material is based upon work supported by the National Science Foundation 
under Grant No. DMS-1928930 while the authors participated in the SRiM program hosted 
by the Simons Laufer Mathematical Sciences Institute (formerly Mathematical 
Sciences Research Institute) in Berkeley, California, during the summer of 2025.

\section{Lower bound}

For the lower bound, it is sufficient to rely on the lower bound of Theorem \ref{ATST} (ATST). 

\begin{definition}
    Given $E \subset \R^2$ compact and $Q \in \mathcal{D}$, define
    \begin{align*}
        r_E(Q): = \begin{cases}
            \displaystyle \inf_{\ell \text{ line}} \sup_{x \in E \cap Q} \mbox{dist}(x, \ell) & \text{ if } E \cap Q \neq \varnothing, \\
            0 & \text{ if } E \cap Q = \varnothing.
        \end{cases}
    \end{align*}
\end{definition}

\begin{remark}
    Note that $r_E(Q)=\beta_E(Q)|Q|$. Because of our proof techniques, $\beta_E(Q)$ will be featured more often than $r_E(Q)$ below. However, the latter is a more relevant quantity as we are dealing with a scale-dependent problem, as opposed to the problem dealt with by Jones in \cite{J}, who first introduced $\beta_E(Q)$. It is interesting to note that $r_E(Q)$ is the same quantity defined by Okikiolu in \cite{O}. 
\end{remark}

\begin{lem}\label{lem:diameter gap}
    Let $E \subset \R^2$ be a compact set and $\Gamma$ be any curve such that $\mathcal{N}(\Gamma,r) \supset E$. Then $|E|-2r\le |\Gamma|$.
\end{lem}
\begin{proof}
     Because $E$ is compact, there exist two points $x_1, x_2 \in E$ such that $\dist(x_1,x_2)=|E|$. 
     Since $E\subset\mathcal{N}(\Gamma,r)$, $ x_1, x_2 \in \mathcal{N}(\Gamma,r)$.  Then 
     \[|E|=\dist(x_1,x_2) \leq \dist(x_1, \Gamma)+|\Gamma|+\dist(x_2,\Gamma)\le r+|\Gamma|+r=|\Gamma|+2r. \qedhere\] 
\end{proof}
We can now prove the lower bound from Theorem \ref{thm: Main}.
\begin{theorem}\label{thm: lower bound}
Given a compact set $E$ in $\R^2$,
\[
|E| - 2r + \sum_{Q \in \mathcal{D}} \max\{\beta_E(Q) - r/|Q|, 0\}^2|Q| \lesssim \Lambda(E, r).
\]
\end{theorem}
\begin{proof}
Let $\Gamma$ be an $r$-maximum distance minimizer of $E$. 
First note for any cube $Q \in \mathcal{D}$, $r_E(Q) \leq r_{\mathcal{N}(\Gamma, r)}(Q) \leq r_\Gamma(Q) + r$.
Hence,
\[
\beta_E(Q) \leq \beta_{\mathcal{N}(\Gamma, r)}(Q) \leq \beta_\Gamma(Q) + \frac{r}{|Q|}.
\]
Thus, 
\[
\max\left\{\beta_E(Q) - \frac{r}{|Q|}, 0\right\} \leq \beta_\Gamma(Q) \quad \text{ for all } Q \in \mathcal{D}.
\]
Thus, applying the Analyst's Traveling Salesperson Theorem to $\Gamma$, we get 
\begin{align*}
|\Gamma| + \sum_{Q \in \mathcal{D}} \max\{\beta_E(Q) - r/|Q|, 0\}^2|Q| &\leq  |\Gamma| + \sum_{Q \in \mathcal{D}} \beta_\Gamma(Q)^2|Q| \\
&\lesssim \mathcal{H}^1(\Gamma) \\
&= \Lambda(E, r).
\end{align*}
Finally, since $|E| - 2r \leq |\Gamma|$, by Lemma \ref{lem:diameter gap}, we get 
\[
|E| - 2r + \sum_{Q \in \mathcal{D}} \max\{\beta_E(Q) - r/|Q|, 0\}^2|Q| \lesssim \Lambda(E,r). \qedhere
\]
\end{proof}
\begin{remark}
    If $|E| <2r$, then $\max\{\beta_E(Q) - r/|Q|, 0\}=0$ for all dyadic cubes $Q$, which makes the lower bound in Theorem \ref{thm: lower bound} negative. However, in this case, the minimizing curve for MDP is a single point, which has $\mathcal{H}^1$ measure 0.
\end{remark}

\section{Upper bound}
We construct a candidate curve $\Gamma^*$ for $\Lambda(E, r)$, when $E\subset \mathbb{R}^2$ is compact. We then show that $\Gamma^*$ satisfies the desired upper estimate in Theorem \ref{thm: Main}.

\subsection{Building convex hulls}
Recall that the convex hull of a set $F$ is the small convex set containing $F$. We use the notation $\conv(F)$ to indicate the convex hull of the set $F$.  We construct a nested family $\mathscr{C}$ of convex hulls that are indexed using binary strings $\sigma\in\bigcup_{i=1}^\infty \{0,1\}^i$.

\subsubsection{Initialization and first generation of construction}
 Initialize $\mathscr{C}_0 = \{C_\varnothing\}$ where $C_\varnothing = \conv(E)$. Fix a diameter line segment $\hat{L}_\varnothing$ given by $\hat{\beta}(C_\varnothing)$.  Let $a_\varnothing$ and $b_\varnothing$ be the endpoints of $\hat{L}_\varnothing$. Because $C_\varnothing$ is the convex hull of $E$, and we are taking a diameter line, observe that $a_\varnothing$ and $b_\varnothing$ are both in $E$. With a mildly improper use of the interval notation, we can write $\hat{L}_\varnothing=[a_\varnothing, b_\varnothing]$. 
Let $\pi_\varnothing$ be the orthogonal projection onto the line  $\hat{L}_\varnothing$. 
Because $E = E\cap C_\varnothing$, observe that $\pi_\varnothing (E) \subset \hat{L}_\varnothing$.
Partition $\hat{L}_\varnothing$ into three intervals of equal length.  Label the the (open) middle interval $K_\varnothing$, and label the remaining (closed) intervals as $I_0\ni a_\varnothing$ and $I_1\ni b_\varnothing$.

We consider two cases 
    \begin{enumerate}
        \item[(P1)] $\pi_\varnothing(E\cap C_\varnothing)\cap K_\varnothing \neq \varnothing$.
        Pick $z_\varnothing \in \pi_\varnothing(E\cap C_\varnothing)\cap K_\varnothing$ and let 
        \[
        C_{0} = \conv(\pi^{-1}_\varnothing([a_\varnothing, z_\varnothing])\cap (E\cap C_\varnothing)) \quad \text{ and } \quad C_{1} = \conv(\pi^{-1}_\varnothing([z_\varnothing, b_\varnothing])\cap (E\cap C_\varnothing)).
        \]

        \item[(P2)] $\pi_\varnothing(E\cap C_\varnothing)\cap K_\varnothing = \varnothing$. 
        Let 
        \[
        C_{0} = \conv(\pi^{-1}_\varnothing(I_0)\cap (E\cap C_\varnothing)) \quad \text{ and } \quad C_{1} = \conv(\pi^{-1}_\varnothing(I_1)\cap (E\cap C_\varnothing)).
        \]     
    \end{enumerate}
    
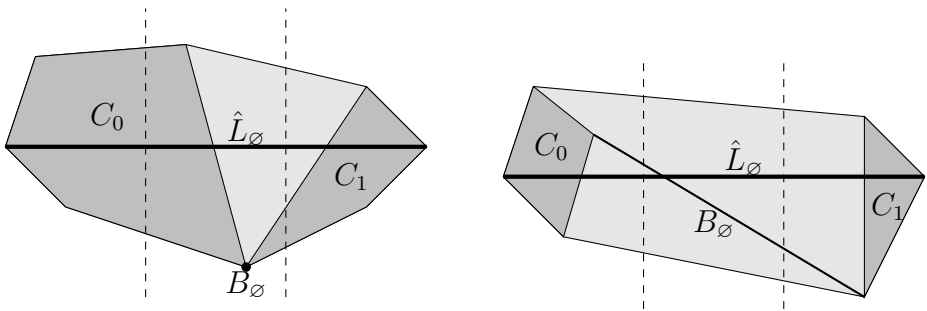
\begin{figure}[h!]
\centering
    \begin{tikzpicture}[scale=0.8]
  \draw[fill=gray!20] (0,0) -- (3,-1) -- (5,0) -- (6,1) -- (5,2) -- (2,2.7) -- (-0.5,2.5) -- (-1,1) -- cycle;
\draw[fill=gray!50] (0,0) -- (3,-1) -- (2,2.7) -- (-0.5,2.5) -- (-1,1) -- cycle;
\draw[fill=gray!50] (3,-1) -- (5,0) -- (6,1) -- (5,2) -- cycle;
\draw[ultra thick] (-1,1) -- (6,1);
 \draw[color=black] (0.7,1.5) node {$C_0$};
  \draw[color=black] (4.75,0.5) node {$C_1$}; 
  \draw[fill=black] (3,-1) circle (2pt);
  \draw[color=black] (3,-1.3) node {$B_\varnothing$};
  \draw[color=black] (3,1.3) node {$\hat L_\varnothing$};
\draw[dashed] (1.33, -1.5) -- (1.33, 3.3);
\draw[dashed] (3.66, -1.5) -- (3.66, 3.3);
\end{tikzpicture}
\qquad
    \begin{tikzpicture}[scale=0.8]
  \draw[fill=gray!20] (0,0) -- (5,-1) -- (6,1) -- (5,2) -- (-0.5,2.5) -- (-1,1) -- cycle;
  \draw[fill=gray!50] (0,0) -- (0.5,1.7) -- (-0.5,2.5) -- (-1,1) -- cycle;
  \draw[fill=gray!50] (5,-1) -- (6,1) -- (5,2) -- cycle;
  \draw[thick] (0.5,1.7) -- (5,-1);
   \draw[color=black] (-0.2,1.5) node {$C_0$};
  \draw[color=black] (5.4,0.5) node {$C_1$}; 
  \draw[color=black] (2.5,0.2) node {$B_\varnothing$};
  \draw[color=black] (3,1.3) node {$\hat L_\varnothing$};
\draw[ultra thick] (-1,1) -- (6,1);
\draw[dashed] (1.33, -1.2) -- (1.33, 3);
\draw[dashed] (3.66, -1.2) -- (3.66, 3);
\end{tikzpicture}
\caption{\label{figure: convex construction} Cases (P1) and (P2) of the first generation of the construction inside $C_\varnothing$.}
\end{figure}
    In either case, let $\mathscr{C}_1=\{C_0,C_1\}$.
    Let $B_\varnothing$ be a shortest line segment connecting a point $e_0\in E\cap C_0$ to a point $e_1\in E\cap C_1$. 
    Note that in case of (P1), $B_\varnothing$ is a single point, and hence has zero length. 
    In the case of (P2), $\dist(C_0, C_1) \geq \frac{1}{3}|\hat{L}_\varnothing|$, and hence $B_\varnothing$ has positive length.

    \subsubsection{Inductive Construction} Suppose $j$ generations of convex hulls and bridges have been constructed. For each $\sigma\in D_j$,
    let $\hat{L}_\sigma$ be a diameter line segment given by $\hat{\beta}(C_\sigma)$. Using the same notation as above, let $\hat{L}_\sigma = [a_\sigma,b_\sigma]$. Now we consider the orthogonal projection onto the line containing $\hat{L}_\sigma$. By construction, $\pi_\sigma(E \cap C_\sigma) \subset \hat{L}_\sigma$. Partition $\hat{L}_\sigma$ into three intervals, $I_{\sigma*0}, K_\sigma, I_{\sigma*1}$. 
    We consider two cases 
    \begin{enumerate}
        \item[(P1)] $\pi_\sigma(E\cap C_\sigma)\cap K_\sigma \neq \varnothing$.
        Pick $z_\sigma \in \pi_\sigma(E\cap C_\sigma)\cap K_\sigma$ and let 
        \[
        C_{\sigma*0} = \conv(\pi^{-1}_\sigma([a_\sigma, z_\sigma])\cap (E\cap C_\sigma)) \text{ and }C_{\sigma*1} = \conv(\pi^{-1}_\sigma([z_\sigma, b_\sigma])\cap (E\cap C_\sigma)).
        \]

        \item[(P2)] $\pi_\sigma(E\cap C_\sigma)\cap K_\sigma = \varnothing$. 
        Let 
        \[
        C_{\sigma*0} = \conv(\pi^{-1}_\sigma(I_{\sigma*0})\cap (E\cap C_\sigma)) \quad \text{ and } \quad C_{\sigma*1} = \conv(\pi^{-1}_\sigma(I_{\sigma*1})\cap (E\cap C_\sigma)).
        \]     
    \end{enumerate}
    Let $B_\sigma$ be a shortest line segment connecting a point $e_{\sigma*0}\in E\cap C_{\sigma*0}$ to a point $e_{\sigma*1}\in E\cap C_{\sigma*1}$. 
    Repeat the same for all other convex hulls in $\mathscr{C}_j$.  Let $\mathscr{C}_{j+1}=\{C_{\sigma*i}: |\sigma|=j, i\in\{0,1\}\}$.

    Following standard terminology, we say $C_{\sigma*i}$ is a {\it child} of $C_\sigma$ for $i=0,1$.

\subsection{Properties of the collection of convex hulls} Before constructing the candidate curve $\Gamma^*$, we need to show a few properties of the collection $\mathscr{C}$.

\begin{definition}[$\hat\beta$-number]
For any convex region $C \subset \mathbb{R}^2$, define 
\[
\hat{\beta}(C) = \sup_{\text{ $\hat{L}$ diameter line segment of }C}\sup_{y \in C}\dist(y, \hat{L}).
\]
\end{definition}

\begin{lem}\label{lem: geometric properties of conves hulls containment}
\begin{enumerate}
    \item[(C1)]$\displaystyle E\subset \bigcup_{|\sigma|=k} C_\sigma$.
    \item[(C2)] If $|\tau|=|\sigma|$, then $C_\tau\cap C_\sigma=\varnothing$.
    \item[(C3)] If $C\in\mathscr{C}$, then there is a triangle with height $\hat{\beta}(C)|C|$ and base $|C|$ contained in $C$.
    \item[(C4)]  Each $C\in\mathscr{C}$ is contained in a rectangle with height $2\hat{\beta}(C)|C|$ and width $|C|$.
    \item[(C5)] $C_{\sigma*i}$, for $i \in \{0,1\}$ is contained in a rectangle with height $2\hat\beta(C_\sigma)|C_\sigma|$ and width $\frac23|C_\sigma|$.
\end{enumerate}
\end{lem}

\begin{proof}
    (C1) and (C2) follow immediately from the construction of the collection $\mathscr{C}$. (C3) follows from the definition of $\hat\beta(C)$ and the convexity of $C$.  (C4) follows from the definition of $\hat\beta(C)$. To prove (C5), recall that $C_{\sigma*i}\subset C_\sigma$ and that in both cases of the construction, $C_{\sigma*i}$ projects onto a portion of $\hat L_{\sigma}$ that is at most $\frac23$ of the total length of $\hat L_{\sigma}$.
\end{proof}

\begin{cor}\label{cor: geometric properties of convex hulls area
} For $C\in \mathscr{C}$.
 \[\frac{1}{2}\hat{\beta}(C)|C|^2\le \area(C)\le 2\hat\beta(C)|C|^2.\]
\end{cor}
\begin{proof}
The lower bound follows from (C3) and the upper bound follows from (C4).
\end{proof}

\begin{figure}[htbp]
\centering
    \begin{tikzpicture}[scale=0.75]
 \draw[fill=gray!50] (0,0) -- (3,-1) -- (5,0) -- (6,1) -- (5,2) -- (2,2.7) -- (-0.5,2.5) -- (-1,1) -- cycle;
  \draw[thick] (-1,1) -- (6,1);
  \draw[dashed] (-1,-1) -- (6,-1) -- (6,3) -- (-1,3) -- cycle;
  \draw[<->] (-1.3,1) -- (-1.3,3);
  \draw[color=black] (-2.4, 2) node {$\hat\beta(C)|C|$};
  \draw[<->] (-1, -1.3) -- (6,-1.3);
  \draw[color=black] (2.5,-1.7) node {$|C|$};
\end{tikzpicture}
 \qquad 
 \begin{tikzpicture}[scale=0.4]
 \draw[color=gray!20, fill=gray!20] (-1.2,-4.13) -- (10.8,3.87) -- (10.8,7) -- (6.35,7) -- (-1.2,2.13) -- cycle;
 \draw[fill=gray!50] (0,0) -- (3,-1) -- (5,0) -- (6,1) -- (5,2) -- (2,2.7) -- (-0.5,2.5) -- (-1,1) -- cycle;
 \draw[fill=gray!50] (6.2,2.7) -- (8,2) -- (10,3.5) -- (7,6) -- cycle;
        \draw[-] (2.8,-1) rectangle (6.8,3);
        \draw[color=black] (7.4,-0.5) node {$Q$};
        \draw[dashed] (-1.2,-5) rectangle (10.8,7);
        \draw[color=black] (11.8, -4) node {$3Q$};
        \draw[thick] (-1.2,-1.13) -- (10.8, 6.87); 
        \draw[color=black] (8.4,6.2) node {\footnotesize$\ell_{3Q}$};
        \draw[<->] (5.45,3.3) -- (3.8,5.4);
        \draw[color=black] (2.5,4) node {\tiny $\beta_E(3Q)|3Q|$};
    \end{tikzpicture}
\caption{\label{figure: different betas} A comparison of $\hat\beta(C)$ and $\beta_E(3Q)$, for $C \in \mathscr{C}(Q)$.}
\end{figure}
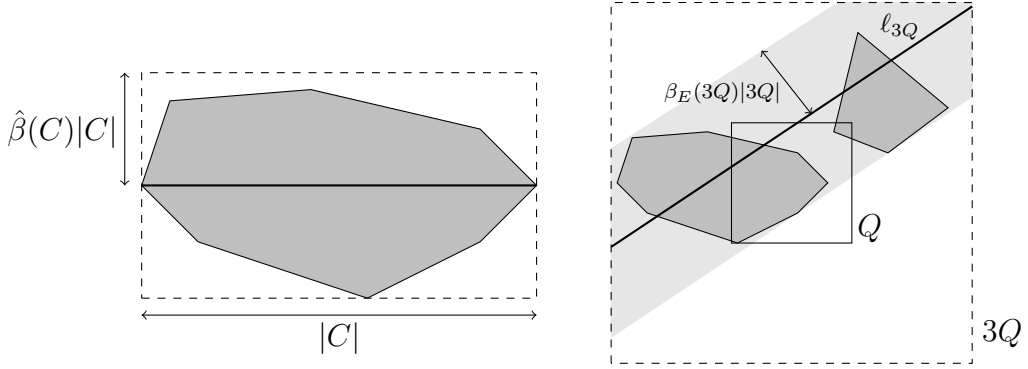

\begin{lem}[{\cite[Lemma 10.5.3]{BP}}]
\label{lem: diameter_decrease}
There is $M<\infty$ such that if $C_\sigma\in \mathscr{C}_j$ and $C_{\sigma*\tau}\in\mathscr{C}_{j+M}$ then $|C_{\sigma*\tau}|\le \frac12 |C_\sigma|.$
\end{lem}

\begin{proof}
Consider the triangle given by Lemma \ref{lem: geometric properties of conves hulls containment} (C3).  The part of the triangle that projects onto $I_{\sigma*1}$ has base length at least $|C_\sigma|/3$ and height at least $\hat{\beta}(C_\sigma)|C_\sigma|/3$.  Thus the area of the portion of the triangle that projects onto $I_{\sigma*1}$ is at least \[\frac{1}{2}\left(\frac{|C_\sigma|}{3}\right)\left(\frac{\hat{\beta}(C_\sigma)|C_\sigma|}{3}\right)=\frac{\hat{\beta}(C_\sigma)|C_\sigma|^2}{18}\ge \frac{\text{area}(C_\sigma)}{36},\]
where the last inequality uses Corollary \ref{cor: geometric properties of convex hulls area
}.
 Since $C_{\sigma*2}$ does not contain any part of $C_\sigma$ that projects onto $I_{\sigma*1}$, 
 \[\text{area}(C_{\sigma*2})\le \frac{35}{36}\text{area}(C_\sigma).\]
 An analogous argument can be used to show that 
 \[\text{area}(C_{\sigma*1})\le \frac{35}{36}\text{area}(C_\sigma).\]

Obverse that for $|\tau|=m$,
 \[\text{area}(C_{\sigma*\tau})\le \left(\frac{35}{36}\right)^m\text{area}(C_\sigma).\]
If follows that 
\begin{align*}
    \frac{1}{2}\hat{\beta}(C_{\sigma*\tau})|C_{\sigma*\tau}|^2\le \text{area}(C_{\sigma*\tau})\le\left(\frac{35}{36}\right)^m\text{area}(C_\sigma)\le \left(\frac{35}{36}\right)^m 2\hat{\beta}(C_\sigma)|C_\sigma|^2.
\end{align*}
If $|C_{\sigma*\tau}|>\frac{1}{2}|C_\sigma|$, then 
\[\frac{1}{8}\hat{\beta}(C_{\sigma*\tau})|C_\sigma|^2\le 2\left(\frac{35}{36}\right)^m\hat{\beta}(C_\sigma)|C_\sigma|^2,\] so 
\[\hat{\beta}(C_{\sigma*\tau})\le 16\left(\frac{35}{36}\right)^m\hat{\beta}(C_\sigma)\le 16\left(\frac{35}{36}\right)^m.\]
However, for $\hat{\beta}(C_{\sigma*\tau})$ sufficiently small, it must be the case that 
\[|C_{\sigma*\tau*i}|< \frac34 |C_{\sigma*\tau}|.\]
To see this, note that, by Lemma \ref{lem: geometric properties of conves hulls containment} (C5),$C_{\sigma*\tau*i}$ is contained in a rectangle of width less than or equal to $\frac23 |C_{\sigma*\tau}|$ and height less than or equal to $2\hat\beta(C_{\sigma*\tau})|C_{\sigma*\tau}|$.  We have 
\[
|C_{\sigma*\tau*i}|\le\sqrt{\left(2\hat\beta(C_{\sigma*\tau})|C_{\sigma*\tau}|\right)^2 + \left(\frac23 |C_{\sigma*\tau}|\right)^2}=\sqrt{\left(2\hat\beta(C_{\sigma*\tau})\right)^2 + \left(\frac23\right)^2}\,|C_{\sigma*\tau}|.
\]  
Therefore, it is enough to choose an $m$ so that \[\sqrt{\left(2\hat\beta(C_{\sigma*\tau})\right)^2 + \left(\frac23\right)^2}\le \sqrt{4\left(16^2\left(\frac{35}{36}\right)^{2m}\right)+\frac{4}{9}}\le \frac{3}{4}.\] Accordingly, we set $m_0=\log\left(\frac{1}{1024}\sqrt{\frac{17}{576}}\right)/\log(35/36)$.
Since $\left(\frac{3}{4}\right)^3<\frac{1}{2}$, after $M=3(m_0+1)$ iterations of construction, it is necessarily the case that the diameter has decreased by a factor of $\frac12$.
\end{proof}

\begin{cor}
\label{cor: convex hulls per point} Fix a dyadic cube $Q$. 
    For any $x\in E\cap Q$, there are at most $M$ convex hulls $C$ such that $C\in\mathcal{C}(Q)$ and  $x\in C$.
\end{cor}

\begin{proof}
    This result follows immediately as a consequence of (C1), (C2), and Lemma \ref{lem: diameter_decrease}.
\end{proof}

\subsection{Labeling convex hulls as good and bad}
To construct a curve $\Gamma^*$ we need to identify the convex hulls which are sufficiently flat (in terms of $r$). We do so by relating convex hulls to dyadic cubes, and looking at the relationship between $r_E(Q)$ and $r$.
     \begin{definition}
      For each dyadic cube $Q$, we say that a convex hull $C$ is associated to $Q$ if $C\cap Q\ne\varnothing$ and $\side(Q)/2 < |C| \leq \side(Q)$. For any cube $Q \in \mathcal{D}$ we call $\mathscr{C}(Q)$ the collection of convex hulls associated to $Q$.
 \end{definition}
   \begin{definition}[good and bad cube]
        We say that a dyadic cube $Q \in \mathcal{D}$ and write $Q \in \mathcal{D}_\text{good}$ is a good cube if $r_E(3Q) < 2r$. We say that $Q$ is bad and write $Q \in \mathcal{D}_\text{bad}$ if it is not good.
          Let $\mathscr{G}= \bigcup_{Q \in \mathcal{D}_\text{good}} \mathscr{C}(Q)$ and $\mathscr{B}= \bigcup_{Q \in \mathcal{D}_\text{bad}} \mathscr{C}(Q)$.
        \end{definition}  
\begin{definition}[good and bad convex hull]
    We say that a convex hull $C$ is good if there exists a good cube $Q$ such that $C \in \mathscr{C}(Q)$. If $C$ is not good, we say that $C$ is bad.
\end{definition} 

For a good convex hull $C_\sigma$, fix one good cube $Q$ with $C_\sigma\in \mathscr{C}(Q)$.  Let $L_\sigma=C_\sigma\cap \ell_\sigma$ where $\ell_\sigma$ is a line satisfying $\sup_{y\in E\cap 3Q}\dist(y, \ell_\sigma)\le 2r$, which is guaranteed to exist by the fact that $\beta_E(3Q)<2r$.

    \subsection{Building candidate curves} \label{subsect: curve}

    We iteratively construct finitely many subsets $\Gamma_0, \dots,\Gamma_N$ of $\mathbb{R}^2$, each of which is a union of line segment and convex hulls.  After construction terminates, we set $\Gamma^* = \Gamma_N$, where $\Gamma_N$ is only a union of finitely many line segments.  We show that $\Gamma^*$ is a candidate belonging to the class of minimizers for $E$.
    
    \subsubsection{Initialization and first generation construction}

    If $C_\varnothing$ is good, let $\mathscr{G}_0 = \{C_\varnothing\}$, $\mathscr{B}_0 = \varnothing$, and let $\Gamma_0 = \partial C_\varnothing \cup L_\varnothing$. In this case, we set $\Gamma^*=\Gamma_0$, and the construction of a candidate curve is complete.
    
    If $C_\varnothing$ is bad, let $\mathscr{G}_0 = \varnothing$ and $\mathscr{B}_0 = \{C_\varnothing\}$ and let $\Gamma_0 = C_\varnothing$.
 
To construct $\Gamma_1$, we look at $C_0,C_1\subset C_\varnothing$. We define $\mathscr{G}_1=\{C_{\varnothing*i}|C_{\varnothing*i}\text{ is good}\}$ and $\mathscr{B}_1=\{C_{\varnothing*i}|C_{\varnothing*i}\text{ is bad}\}$. Then we define
   \[
    \Gamma_1 =\left(\bigcup_{C_\tau \in \mathscr{G}_1}\partial C_\tau \cup L_\tau \right)\cup \left(\bigcup_{C_\sigma \in \mathscr{B}_1} C_\sigma\right)\cup B_{\varnothing}.
     \]
\subsubsection{Inductive Construction}
Let $j \geq 1$. Suppose that $\mathscr{G}_j$ and $\mathscr{B}_j$ have been established and $\Gamma_j$ defined by 
     \[
    \Gamma_j = \left(\bigcup_{k=1}^j\bigcup_{C_\tau \in \mathscr{G}_k}\partial C_\tau \cup L_\tau\right)\cup \left(\bigcup_{C_\sigma \in \mathscr{B}_j} C_\sigma\right)\cup \left(\bigcup_{k=0}^{j-1}\bigcup_{\tau : C_\tau \in \mathscr{B}_k}B_{\tau}\right).
    \]
    For each $C_\sigma \in \mathscr{B}_j$, we define $\mathscr{G}_{j+1} = \left\{C_{\sigma \ast i}\mid C_{\sigma \ast i} \text{ is good}\right\}$, and $\mathscr{B}_{j+1} = \left\{C_{\sigma \ast i}\mid C_{\sigma \ast i} \text{ is bad}\right\}$.
        Intuitively, $\mathscr{G}_{j+1}$ (resp.~$\mathscr{B}_{j+1}$) contains the ``good (resp.~bad) convex children of bad convex parents.''
        Note that we do not continue the construction inside the convex hulls in $\mathscr{G}_j$, so children of $C_\sigma\in\mathscr{G}_j$ are contained in neither $\mathscr{G}_{j+1}$ nor $\mathscr{B}_{j+1}$.
 
    We define 
     \begin{align*}
    \Gamma_{j+1} & \! = \! \left(\bigcup_{k=1}^j\bigcup_{C_\tau \in \mathscr{G}_k}\! \! \partial C_\tau \cup L_\tau\right) \! \cup \! \left(\bigcup_{C_\tau \in \mathscr{G}_{j+1}} \! \!\partial C_\tau \cup L_\tau\right) \! \cup \! \left(\bigcup_{C_\sigma \in \mathscr{B}_{j+1}} C_\sigma\right)\! \cup \!\left(\bigcup_{k=0}^{j-1}\bigcup_{\tau : C_\tau \in \mathscr{B}_k}B_{\tau}\right) \! \cup \! \!\bigcup_{C_\tau \in \mathscr{B}_{j}}B_\tau\\
    & = \left(\bigcup_{k=1}^{j+1}\bigcup_{C_\tau \in \mathscr{G}_k}\partial C_\tau \cup L_\tau\right)\cup \left(\bigcup_{C_\sigma \in \mathscr{B}_{j+1}} C_\sigma\right)\cup \left(\bigcup_{k=0}^{j}\bigcup_{\tau : C_\tau \in \mathscr{B}_k}B_{\tau}\right).
    \end{align*}
        \begin{figure}[htbp]
\centering
     \begin{tikzpicture}[scale=0.92]
  \draw[-] (0,0) -- (5,-1) -- (8,1) -- (5,2) -- (-0.5,2.5) -- (-0.5,1) -- cycle;
  \draw[color=purple] (0,0) -- (0.2,1.7) -- (-0.5,2.5) -- (-0.5,1) -- cycle;
  \draw[color=purple] (-0.5,2.5) -- (0,0);
  \draw[color=purple, fill=purple!20] (5,-1) -- (8,1) -- (5,2) -- cycle;
  \draw[color=purple] (0.2,1.7) -- (5,-1);
\end{tikzpicture}
\qquad
    \begin{tikzpicture}[scale=0.92]
  \draw[-] (0,0) -- (5,-1) -- (8,1) -- (5,2) -- (-0.5,2.5) -- (-0.5,1) -- cycle;
  \draw[color=purple] (0,0) -- (0.2,1.7) -- (-0.5,2.5) -- (-0.5,1) -- cycle;
 \draw[color=purple] (-0.5,2.5) -- (0,0);
  \draw[color=purple] (5,-1) -- (5.2,1) -- (5,2) -- cycle;
  \draw[color=purple, fill=purple] (5,2) circle (1pt);
  \draw[color=purple] (0.2,1.7) -- (5,-1);
  \draw[color=purple, fill=purple!20] (5,2) -- (6.5,0.2) -- (8,1) -- cycle;
\end{tikzpicture}
\caption{\label{figure: curve construction} Two steps of the construction of the $\Gamma_j$'s (shown in purple).}
\end{figure}
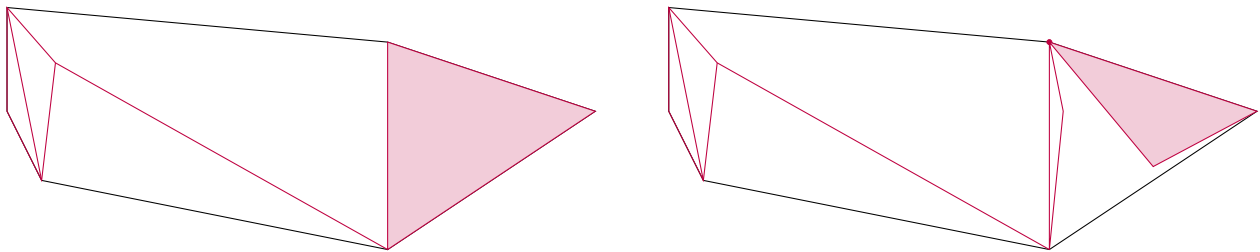

    \subsubsection{Termination of construction}
By Lemma \ref{lem: diameter_decrease}, after finitely many, $N$, (where $N$ is bounded above in terms of $r$) iterations of the construction, $|C_\sigma|\le\frac{r}{3}$ for all convex hulls $C_\sigma$, $\sigma\in\bigcup_{i=1}^\infty \{0,1\}^i$. Then there is $Q$ such that $C_\sigma\cap Q\ne\varnothing$ and
    \[\side(Q)/2\le |C_\sigma|< \side(Q).\]
   Since $\side(Q)\le\frac{r}{3}$, it is necessarily the case that $r_E(3Q)\le r$ because $r_E(3Q)$ is bounded by $\side(3Q)$.  
   Therefore $\mathscr{B}_N = \varnothing$. Then we define
   \[
   \Gamma^* = \Gamma_N = \left(\bigcup_{k=1}^N\bigcup_{C_\tau \in \mathscr{G}_k}\partial C_\tau \cup L_\tau\right)\cup  \left(\bigcup_{k=0}^{N-1}\bigcup_{\tau : C_\tau \in \mathscr{B}_k}B_{\tau}\right).
   \]
  
\subsection{Verifying that \texorpdfstring{$\Gamma^\ast$}\ \ is a candidate curve.}
The class of objects that MDP minimizes on is curves of finite length whose $r$-neighborhood covers $E$. 
By construction, $\Gamma^*$ has finite length, as it is the union of finitely many line segments. To show that it is in fact a curve, we must show it is connected.

\begin{lem}\label{rem: connected and finite length}
$\Gamma^*$ is connected. 
\end{lem}
\begin{proof}
To see that $\Gamma^*$ is connected, suppose that $C_\sigma\in \mathscr{B}_j$ and $B_{\tau}$ is a bridge with an endpoint $e_\tau$ in $E\cap C_\sigma$. To construct $\Gamma_{j+1}$ from $\Gamma_{j}$, we first remove $C_\sigma$ and replace it with $C_{\sigma*0}\cup B_\sigma\cup C_{\sigma*1}$, so that $C_{\sigma*0}$ is connected to $C_{\sigma*1}$ via $B_\sigma$.  
Because $ e_\tau\in C_{\sigma*0}\cup C_{\sigma*1}$, we get that $\left(\Gamma_j\setminus C_\sigma\right) \cup C_{\sigma*0}\cup B_\sigma\cup C_{\sigma*1}$ is connected. 

If both children of $C_\sigma$ are bad, then we move on to the next bad convex hull.
If a child, say $C_{\sigma\ast 0}$, is good, then we remove $C_{\sigma\ast 0}$ and replace it with $\partial C_{\sigma\ast 0} \cup L_{\sigma\ast 0}$.
For any bridge $B_\tau$ with an endpoint $e_\tau \in E\cap C_{\sigma\ast 0}$, the bridge $B_\tau$ will intersect $\partial C_{\sigma\ast 0}$. 
Thus, since $\partial C_{\sigma\ast 0} \cup L_{\sigma\ast 0}$ is connected, the replacement $(\Gamma_j\setminus C_{\sigma*\ast 0}) \cup (\partial C_{\sigma\ast 0} \cup L_{\sigma\ast 0})$ is also connected.
\end{proof}

   \begin{lem}\label{lem: ab}
     Let $C$ be a convex set in $\mathbb{R}^2$ and let $\ell$ be a line in $\mathbb{R}^2$.
     \begin{enumerate}
     \item[(a)] If $\sup_{x \in C}\dist(x, \ell) \leq r$ then $\mathcal{N}(\partial C, r) \supset C$.
     \item[(b)] If $\sup_{x \in C}\dist(x, \ell) \leq 2r$ then $\mathcal{N}(\partial C \cup (C\cap \ell), r) \supset C$.
     \end{enumerate}
     \end{lem}
     \begin{proof}
     \begin{enumerate}
     \item[(a)] Let $\pi$ be the orthogonal projection of $\mathbb{R}^2$ onto $\ell$. 
     For any $x \in \ell$, let $\ell^\perp_x \cap C$ by the line perpendicular to $\ell$ containing $x$.
     By assumption $\sup_{y \in \ell_x^\perp\cap C} |x - y| \leq r$. 
     If $\ell_x^\perp \cap C\cap H^{\pm}$ is not empty, then $(\ell_x^\perp \cap \partial C)\cap H^{\pm}$ is also not empty. 
     Then for any point $b_x^{\pm} \in (\ell_x^\perp \cap \partial C)\cap H^{\pm}$, we have that the ball $B(b_x^{\pm}, r)$ contains $(\ell_x^\perp \cap \partial C)\cap H^{\pm}$. 
     Therefore
     \[
     \mathcal{N}(\partial C, r) = \bigcup_{b \in \partial C} B(b, r) \supset \bigcup_{x \in \ell} B(b_x^- \cup b_x^{+}, r) \supset \bigcup_{x \in \ell}\ell_x^\perp \cap C = C.
     \]
    
     \item[(b)] First, let $H^+$ and $H^-$ be the two half-spaces in $\mathbb{R}^2$ with $\partial H^{\pm} = \ell$, let $\boldsymbol{n}$ be the unit normal vector to $\ell$ pointing in the direction of $H^+$, and let $C^\pm = C\cap H^\pm$. 

     Then, the $r$-neighborhood $\mathcal{N}(\ell + r\boldsymbol{n}, r)$ contains $C^+$, and similarly, the $r$-neighborhood $\mathcal{N}(\ell - r\boldsymbol{n}, r)$ contains $C^-$.
     Since $C^\pm$ are convex sets, part (a) implies  $\mathcal{N}(\partial C^\pm, r) \supset C^\pm$.
     Thus
     \[
     \mathcal{N}(\partial C^+ \cup \partial C^-, r) = \mathcal{N}(\partial C^+, r) \cup \mathcal{N}(\partial C^- ,r) \supset C^+ \cup C^- = C.
     \]
     Since $\partial C^+ \cup \partial C^- = \partial C \cup (C\cap \ell)$, we are done. \qedhere
     \end{enumerate}
     \end{proof}

\begin{prop}
    $\Gamma^*$ is a candidate curve for MDP.
\end{prop}
\begin{proof} 
By Lemma \ref{rem: connected and finite length}, it suffices to show that $\mathcal{N}(\Gamma^*,r) \supset E$.  To this end, let $x\in E$.  
By construction, for each $k\in\mathbb{N}$, there is a $\tau_k$ such that $|\tau_k|=k$ and $x\in C_{\tau_k}$. 
Furthermore, there exists some $k_x$ such that $C_{\tau_{k_x}}\in \mathscr{G}_{{k_x}}$.  Recall that $C_{\tau_{k_x}}\in \mathscr{G}_{\tau_{k_x}}$ implies $C_{\tau_{k_x}}\subset 3Q$ for some $Q$ with $r_E(3Q)<2r$. 
Since $E\cap 3Q\subset \ell_{\tau_{k_x}}\times (-2r,2r)$, it is also the case that $C_{\tau_{k_x}} \subset \ell_{\tau_{k_x}}\times (-2r,2r)$.  
By Lemma \ref{lem: ab}, we have $\mathcal{N}(\partial C_{\tau_{k_x}} \cup L_{\tau_{k_x}},r)\supset C_{\tau_{k_x}}$, and it follows that $x\in \mathcal{N}(C_{\tau_{k_x}} \cup L_{\tau_{k_x}},r)$.   
Note that 
\begin{align*}
\mathcal{N}(\Gamma^*, r)&=\mathcal{N}\left(\left(\bigcup_{k=1}^N\bigcup_{C_\tau \in \mathscr{G}_k}\partial C_\tau \cup L_\tau\right)\cup  \left(\bigcup_{k=0}^{N-1}\bigcup_{\tau : C_\tau \in \mathscr{B}_k}B_{\tau}\right), r\right)\supset \mathcal{N}(C_{\tau_{k_x}} \cup L_{\tau_{k_x}}, r).
\end{align*}
Thus $x\in \mathcal{N}(\Gamma^*, r)$.  
Since $x$ was chosen arbitrarily, we conclude that $E\subset \mathcal{N}(\Gamma^*, r)$.
\end{proof}

\subsection{Estimating Length of the Candidate}   
Finally we show that the curve $\Gamma^*$ has length bounded above by the left hand side of Theorem \ref{thm: Main}.
\begin{lem}[{\cite[Lemma 10.5.2]{BP}}]
\label{lem: bounding new diameters with old diamters}
Suppose that $C_{\sigma\ast 1}$ and $C_{\sigma\ast 2}$ are both bad children of $C_\sigma$. 
If $C_\sigma$ is replaced by $C_{\sigma*1}$, $B_\sigma$, and $C_{\sigma*2}$, then there exists $K>0$ such that
\[|C_{\sigma*0}|+\frac{1}{2}|B_\sigma|+|C_{\sigma*1}|\le |C_\sigma|+K\hat\beta^2(C_\sigma)|C_\sigma|.\]
\end{lem}

\begin{proof}
\begin{enumerate}
  \item  First suppose that $C_{\sigma*0}$, $C_{\sigma*1}$, and $B_\sigma$ were constructed via Case 1, i.e., $|B_{\sigma}|=0$ and $|[a_\sigma, z_\sigma]|, |[z_\sigma, b_\sigma]|$ are both greater than or equal to that $\frac{|C_\sigma|}{3}$. 
    Then $C_{\sigma*0}$ is contained in the rectangle $[a_\sigma, z_\sigma]\times 2\hat\beta(C_\sigma)|C_\sigma|$.  The diameter of this rectangle is 
    \[
    \sqrt{|[a_\sigma, z_\sigma]|^2+(2\hat\beta(C_\sigma)|C_\sigma|)^2}.
    \]
    If  $ \frac{2\hat\beta(C_\sigma)|C_\sigma|}{|[a_\sigma, z_\sigma]|}  < 1$ we can use Taylor approximation to obtain
\begin{align*}
    \sqrt{|[a_\sigma, z_\sigma]|^2+(2\hat\beta(C_\sigma)|C_\sigma|)^2} & = |[a_\sigma, z_\sigma]|  \sqrt{1+\left(\frac{2\hat\beta(C_\sigma)|C_\sigma|}{|[a_\sigma, z_\sigma]|}\right)^2} \\
    & \leq |[a_\sigma, z_\sigma]| \left(1 + K_1\left(\frac{2\hat\beta(C_\sigma)|C_\sigma|}{|[a_\sigma, z_\sigma]|}\right)^2\right) \\
    & \leq |[a_\sigma, z_\sigma]| + 6^2 K_1\hat\beta(C_\sigma)^2|C_\sigma|.
\end{align*}

If $\frac{2\hat\beta(C_\sigma)|C_\sigma|}{|[a_\sigma, z_\sigma]|}>1$, then  $\hat{\beta}(C_\sigma) > \frac{|[a_\sigma, z_\sigma]|}{2|C_\sigma|} \geq \frac16$, which implies that $2\hat{\beta}(C_\sigma) \leq 12 \hat{\beta}(C_\sigma)^2$. Then 
\[ 
\sqrt{|[a_\sigma, z_\sigma]|^2+(2\hat\beta(C_\sigma)|C_\sigma|)^2} \leq  |[a_\sigma, z_\sigma]|  + 2\hat\beta(C_\sigma)|C_\sigma| \leq  |[a_\sigma, z_\sigma]|+ 12\hat\beta^2(C_\sigma)|C_\sigma|,
\]
because $\sqrt{a^2+b^2} \leq a + b$, if $a,b \geq 0$.

The same is true for $C_{\sigma*1}$ using the line segment $[z_\sigma, b_\sigma]$. Putting it together we obtain
\[
|C_{\sigma*0}|+|C_{\sigma*1}|\le |C_\sigma|+ 2(6^2K_1 + 12)\hat\beta^2(C_\sigma)|C_\sigma|,
\]
because $|[a_\sigma,z_\sigma]| + |[z_\sigma,b_\sigma]|=|[a_\sigma,b_\sigma]| \leq |C_\sigma|$.

  \item  Next suppose that $C_{\sigma*1}$, $C_{\sigma*2}$, and $B_\sigma$ were constructed via Case 2.

    If $\hat\beta(C_\sigma)\ge \frac{1}{30}$, then 
    \[|C_{\sigma*0}|+|C_{\sigma*1}|+\frac{1}{2}|B_\sigma|\le \frac{5}{2}|C_\sigma|\le |C_\sigma|+1350\hat\beta^2(C_\sigma)|C_\sigma|,\]
    where $1350=3/2*(30^2)$.

    On the other hand, if $\hat\beta(C_\sigma)<\frac{1}{30}$, because $\sqrt{a^2+b^2} \leq a + b$, if $a,b \geq 0$, we have
    \[
    |C_{\sigma*0}| \leq |I_0| + 2 \hat{\beta}(C_\sigma)|C_\sigma|
    \]
    the analogous for $C_{\sigma*0}$, and
    \[
    |B_\sigma| \leq |C_\sigma| - |I_0| - |I_1| + 2 \hat{\beta}(C_{\sigma})|C_\sigma|.
    \]
    Then
    \begin{align*}
        |C_{\sigma*0}|+|C_{\sigma*1}|+\frac{1}{2}|B_\sigma| & \leq |I_0| + |I_1| + 4 \hat{\beta}(C_\sigma)|C_\sigma| + \frac12 \left(|C_\sigma| - |I_0| - |I_1| + 2 \hat{\beta}(C_{\sigma})|C_\sigma|\right) \\
        & = \frac12 |C_\sigma| + \frac12 |I_0| + \frac12 |I_1| + 5 \hat{\beta}(C_{\sigma})|C_\sigma| \\
        & \leq \frac12 |C_\sigma| + \frac16 |C_\sigma| + \frac16 |C_\sigma| + \frac16 |C_\sigma| \\
        & = |C_\sigma|.
    \end{align*}
    \end{enumerate}
    Choosing $K = \max \{2(36^2K_1 + 12), 1350\}$ the proof is complete.
\end{proof}

\begin{lem}
\label{lem: bounding boundary of convex hull with diameter}
    For a convex hull $C_\sigma$
    \[
    \mathcal{H}^1(\partial C_{\sigma }\cup \hat{L}_{\sigma}) \le 7\,|C_{\sigma}|.
    \]
\end{lem}
\begin{proof}
First, 
\[\mathcal{H}^1(\partial C_{\sigma}\cup \hat L_{\sigma}) \le \mathcal{H}^1(\partial C_{\sigma}) + |C_{\sigma}|.
\]

Since $C_{\sigma}$ is contained in two squares of side length $|C_{\sigma}|$ which share a side equal to $\hat L_{\sigma}$, we have $\mathcal{H}^1(\partial C_{\sigma}) \leq 6\, |C_{\sigma}|$.
\end{proof}

Next, we follow \cite{BP} to compare the quantities $\hat\beta(C)$ to the quantity $\beta(Q)$ for all convex hulls $C$ that are associated to a cube $Q$.  While the proof of the lemma below follows exactly as in \cite{BP}, we note the collection of convex hulls $\mathscr{C}(Q)$ here may be strictly smaller than the collection of convex hulls summed over in \cite{BP} due to the fact that we associate each convex hull to a single dyadic cube.  
\begin{lem}[{\cite[Lemma 10.5.4]{BP}}]
\label{lem:BP1054}
    \[\sum_{C\in\mathscr{C} (Q)}\hat\beta^2(C)|C|\leq (288M)^2\left(\beta_E(3Q)\right)^2|Q|,\]
    where $M$ is as in Lemma \ref{lem: diameter_decrease}.
\end{lem}

\begin{proof}
Using the definition of $\beta_E(3Q)$, fix some minimizing line $\ell_{3Q}$ so that \[\dist(x,\ell_{3Q})\le \beta_E(3Q)|3Q|\qquad\text{for all } x\in E,\]
and then set $W:=\ell_{3Q}\times [-r_E(3Q), r_E(3Q)]\cap 3Q$. This is the strip shown on the left in Figure \ref{figure: different betas}. Let $C\in\mathscr{C}(Q)$.  Recall that by definition of $\mathscr{C}(Q)$, $C\cap Q\ne\varnothing$ and $\side(Q)/2\le |C|\le \side (Q)$ which implies that $C\subset 3Q$.   More precisely, $W$ is convex, $E\cap 3Q\subset W$, and $C$ is a convex hull of a subset of points in $E\cap 3Q$, $C\subset W$ which means that 
\[\area(C)\le \area(W)\le 2r_E(3Q)|3Q|=18\beta_E(3Q)|Q|^2.\] 
On the other hand, by Corollary \ref{cor: geometric properties of convex hulls area
} and $|C|\ge \side(Q)/2$,
\[\area(C)\ge \frac{1}{2}\hat\beta(C)|C|^2\ge \frac{1}{16}\hat\beta(C)|Q|^2.\]
By Corollary \ref{cor: convex hulls per point} point $x\in W$  is in no more $M$ convex hulls.  Thus
\begin{align*}
    \sum_{C \in \mathscr{C}(Q)} \hat \beta(C)|Q|^2 & \le 16 \sum_{C \in \mathscr{C}(Q)} \area(C) \leq 16 M \area(3Q\cap W) \leq 288 M\beta_E(3Q)|Q|^2.
\end{align*}
By dividing the string of inequalities by $|Q|^2$ and applying Cauchy-Schwartz Inequality, we obtain 
\[
\sum_{C \in \mathscr{C}(Q)} \hat \beta(C)^2 \leq \left(\sum_{C \in \mathscr{C}(Q)} \hat \beta(C)\right)^2 \leq (288M)^2 \beta_E(3Q)^2.
\]
Finally, since $C\in\mathscr{C}(Q)$ implies $|C|\le \side(Q)\le |Q|$,
\[
\sum_{C \in \mathscr{C}(Q)} \hat \beta(C)^2 |C|\leq\sum_{C \in \mathscr{C}(Q)} \hat \beta(C)^2 |Q|\leq (288M)^2 \beta_E(3Q)^2 |Q|. \qedhere
\]
\end{proof}

\begin{cor}
    If $Q$ is a bad cube, then 
    \[\sum_{C\in\mathscr{C}(Q)}\hat\beta^2(C)|C|\le (288M)^2(\beta_E(3Q)-r/|3Q|)^2|Q|.\]
\end{cor}

\begin{proof}
    Recall $r_{3Q} = \beta_E(3Q)|3Q|$.
  Because $Q$ is a bad cube, $r_E(3Q)>2r$; that is, $r_E(3Q)/2 > r$. It follows that
    \[r_E(3Q) - r > \frac{r_E(3Q)}{2}.\]
Dividing both sides by $|3Q|$, we may write the inequality in terms of $\beta$-numbers as \[\beta_E(3Q) - \frac{r}{|3Q|} > \frac{\beta_E(3Q)}{2},\]
and a rearrangement yields
\[\beta_E(3Q)  < 2 \left(  \beta_E(3Q) - \frac{r}{|3Q|} \right).\]
Finally, we invoke Lemma \ref{lem:BP1054} to conclude that 
\[
\sum_{C \in \mathscr{C}(Q)} \hat \beta(C)^2 |C|\leq (288M)^2 \beta_E(3Q)^2 |Q|\le (288M)^2\left(\beta_E(3Q)-\frac{r}{|3Q|}\right)^2|Q|. \qedhere
\]
\end{proof}

     \begin{theorem}
     Given a set $E$ in $\R^2$,
        \[ 
        \Lambda(E,R) \lesssim |E|  + \sum_{Q \in \mathcal{D}} \max\{\beta_E(Q) - r/|Q|, 0\}^2|Q|.
        \]
     \end{theorem}  
\begin{proof}
The estimates in Lemma \ref{lem: bounding new diameters with old diamters} do not depend on whether a convex hull is good or bad. This implies that
\[\sum_{C_{\sigma*i}\in \mathscr{G}_{j+1}}|C_{\sigma*i}|+\sum_{C_{\sigma*i}\in\mathscr{B}_{j+1}}|C_{\sigma*i}|+\frac{1}{2}\sum_{|\sigma|=j+1}|B_\sigma|\le \sum_{C_\sigma\in\mathscr{B}_{j}}|C_\sigma|+\sum_{C_\sigma\in\mathscr{B}_j}\beta^2(C_\sigma)|C_\sigma|.\]
Here we are using the fact that we only add convex hulls to the set $\mathscr{G}_{j+1}$ and $\mathscr{B}_{j+1}$ if their parent is bad.
We can rearrange this as
\[\sum_{C_{\sigma*i}\in \mathscr{G}_{j+1}}|C_{\sigma*i}|+\sum_{C_{\sigma*i}\in\mathscr{B}_{j+1}}|C_{\sigma*i}|+\frac{1}{2}\sum_{|\sigma|=j+1}|B_\sigma|- \sum_{C_\sigma\in\mathscr{B}_{j}}|C_\sigma|\le\sum_{C_\sigma\in\mathscr{B}_j}\beta^2(C_\sigma)|C_\sigma|.\]
Adding up over all $N$ generations yields
\[\sum_{j=1}^{N}\sum_{C_{\sigma*i}\in \mathscr{G}_{j+1}}\! \! |C_{\sigma*i}|+\sum_{j=1}^N\sum_{C_{\sigma*i}\in\mathscr{B}_{j+1}}\! \!|C_{\sigma*i}|+\sum_{j=1}^N\frac{1}{2}\sum_{|\sigma|=j+1}|B_\sigma|- \sum_{j=1}^N\sum_{C_{\sigma\in\mathscr{B}_{j}}}\! \!  |C_\sigma|\le\sum_{j=1}^N\sum_{C_\sigma\in\mathscr{B}_j}\beta^2(C_\sigma)|C_\sigma|.\]
Next, we telescope the bad convex hull terms, which reduces the estimate to 
\[\sum_{j=1}^{N}\sum_{C_{\sigma*i}\in \mathscr{G}_{j+1}}|C_{\sigma*i}|+\sum_{C_{\sigma*i}\in\mathscr{B}_{N+1}}|C_{\sigma*i}|+\sum_{j=1}^N\frac{1}{2}\sum_{|\sigma|=j+1}|B_\sigma|-|E|\le\sum_{j=1}^N\sum_{C_\sigma\in\mathscr{B}_j}\beta^2(C_\sigma)|C_\sigma|.\]
However, $\mathscr{B}_{N+1}=\varnothing$.  Therefore we get
\[\sum_{j=1}^{N}\sum_{C_{\sigma*i}\in \mathscr{G}_{j+1}}|C_{\sigma*i}|+\frac{1}{2}\sum_{j=1}^N\sum_{|\sigma|=j+1}|B_\sigma|\le|E|+\sum_{j=1}^N\sum_{C_\sigma\in\mathscr{B}_j}\beta^2(C_\sigma)|C_\sigma|.\]
Finally, using Lemma \ref{lem: bounding boundary of convex hull with diameter}
\begin{align*}
    \mathcal{H}^1(\Gamma^*)&=\sum_{j=1}^N\sum_{C_\sigma\in\mathscr{G}_k}(\mathcal{H}^1(\partial C_{\sigma}\cup (C_{\sigma}\cap L_{\sigma})))+\sum_{j=1}^N\sum_{|\sigma|=j}|B_\sigma|\\
    &\le 7\left(\sum_{j}^{N}\sum_{C_{\sigma}\in \mathscr{G}_{j}}|C_{\sigma}|+\frac{1}{2}\sum_{j=1}^N\sum_{|\sigma|=j+1}|B_\sigma|\right)\\
    &\le 7\left(|E|+\sum_{j=1}^N\sum_{C_\sigma\in\mathscr{B}_j}\beta^2(C_\sigma)|C_\sigma|\right)\\
    &= 7\left(|E|+\sum_{C_\sigma\in\cup_{j=1}^N\mathscr{B}_j}\beta^2(C_\sigma)|C_\sigma|\right)\\
        &\le 7\left(|E|+\sum_{Q\in\mathcal{D}_\text{bad}}\sum_{C_\sigma\in\mathscr{C}(Q)}\beta^2(C_\sigma)|C_\sigma|\right)\\
                &\le 14(288M)^2\left(|E|+\sum_{Q\in \mathcal{D}_\text{bad}}\sum_{C_\sigma\in\mathscr{C}(Q)}(\beta_E(Q)-r/|Q|)^2|Q|\right)\\
    & \le 14(288M)^2\left(|E|+\sum_{Q\in \mathcal{D}_\text{bad}} \max\{\beta_E(Q) - r/|Q|, 0\}^2|Q|\right).
\end{align*}
Note that because we are taking the maximum of $\beta_E(Q)-r/|Q|$ and $0$, we are only adding up over the bad cubes.
\end{proof}
By looking at the asymptotics for $r \to 0$, we recover the classical estimate of the Analyst's Traveling Salesperson Theorem \cite{J}.
\begin{cor}
    If $E$ is contained in a connected set of finite length, $\Gamma^*$, 
    \[
    |E| +  \sum_{\substack{Q \in \mathcal{D}}} \beta^2_E(3Q) |Q| \lesssim \mathcal{H}^1(\Gamma^*) \lesssim |E| +  \sum_{\substack{Q \in \mathcal{D}}} \beta^2_E(3Q) |Q|.
    \]
\end{cor}

\begin{remark}
    If $r=0$, our constructive algorithm recovers exactly the algorithm in \cite{BP} to construct a curve for the Analyst's Traveling Salesperson Theorem \cite{J}. In such case all convex hulls would be good, so our construction would not stop after finitely many steps.

    For a fixed $r>0$, let $\Gamma^*_r$ be the curve constructed in Section \ref{subsect: curve}. If $\Gamma^*$ is the curve constructed in \cite[Section 10.5]{BP}, we have that $\dist_{H}(\Gamma^*, \Gamma^*_r)\le 2r$, where $\dist_H(\cdot, \cdot)$ denotes the Hausdorff distance. By letting $r \to 0$, the curves $\Gamma^*_r$ converge to $\Gamma^*$.
\end{remark}


\begin{thebibliography}{99}

\bibitem[ACMN]{ACMN}
E. G. Alvarado, L.Catalano, T. Merch\'an, L. Naples, \textit{Asymptotics of maximum distance minimizers}, arXiv preprint arXiv:2309.08055 (2023).

\bibitem[AKV]{AKV} E. G. Alvarado, B. Krishnamoorthy, K. R. Vixie, \textit{The Maximum Distance Problem and Minimum Spanning Trees}, Int. J. Anal. Appl., 19 (5) (2021), 633--659.

\bibitem[BP]{BP}
Bishop, Christopher J.; Peres, Yuval. Fractals in probability and analysis. Cambridge Studies in Advanced Mathematics, 162. {\it Cambridge University Press}, Cambridge, 2017. ix+402 pp. ISBN: 978-1-107-13411-9 MR3616046

\bibitem[BOS]{BOS} G. Buttazzo, E. Oudet, and E. Stepanov. 
\textit{Optimal transportation problems with free Dirichlet regions}. Variational Methods for Discontinuous Structures (2002), 41--65.

\bibitem[J]{J}
P. W. Jones, 
\textit{Rectifiable sets and traveling salesman problem}, Invent. Math.  102 (1990), 1--15.

\bibitem[O]{O}
K. Okikiolu, \textit{Characterization of subsets of rectifiable curves in $\mathbb{R}^n$},  J. London Math. Soc. (2) 46 (1992), 336-348.

\end{thebibliography}
\end{document}